\documentclass[a4paper, 12pt]{amsproc}

\usepackage{amsmath,amssymb,amsthm,amsfonts}
\usepackage{mathtools}
\usepackage{latexsym}
\usepackage{mathrsfs}
\usepackage{extarrows}
\usepackage{enumerate} 
\usepackage{enumitem} 
\usepackage{bm}
\usepackage{arydshln}
\usepackage[hypertexnames=false,setpagesize=false]{hyperref}
\usepackage{cite}
\usepackage{bxpapersize}
\usepackage{graphicx,color}
\usepackage[svgnames]{xcolor} 
\usepackage[normalem]{ulem} 
\usepackage{mathrsfs}
\usepackage[dvips,all]{xy}
\usepackage{array, booktabs} 
\usepackage{colortbl}
\usepackage{tabularx} 
\usepackage{float} 
\usepackage{subcaption}
\usepackage{longtable}
\usepackage{tikz}
\usetikzlibrary{intersections, calc, arrows.meta}
\usepackage{hhline}
\makeatletter
\renewcommand{\theenumi}{{\upshape (\@roman\c@enumi)}}

\renewcommand{\p@enumii}{}
\renewcommand{\theenumii}{{\upshape (\@alph\c@enumii)}}

\makeatother
\allowdisplaybreaks[4]
\DeclarePairedDelimiter{\abs}{\lvert}{\rvert}
\usepackage{multicol}
\usepackage{fullpage}
\makeatletter
\@namedef{subjclassname@2020}{%
 \textup{2020} Mathematics Subject Classification}
\makeatother
\usepackage[capitalize,nameinlink,noabbrev,nosort]{cleveref}
\hypersetup{
						colorlinks=true,    
						linkcolor=brown,    
						citecolor=brown,    
						filecolor=brown,    
						urlcolor=brown,     
						}
\usepackage{autonum}
 
\crefname{equation}{}{} 
\numberwithin{equation}{section} 
\newtheorem{theoremcounter}{}[section] 
\newtheorem{dfn}[theoremcounter]{Definition}
\crefname{dfn}{Definition}{Definitions} 
\newtheorem{dfn*}[appendixcounter]{Definition}
\crefname{dfn*}{Definition}{Definitions} 
\newtheorem{thm}[theoremcounter]{Theorem}
\crefname{thm}{Theorem}{Theorems}
\newtheorem{thm*}[appendixcounter]{Theorem}
\crefname{thm*}{Theorem}{Theorems} 
\newtheorem{lem}[theoremcounter]{Lemma}
\crefname{lem}{Lemma}{Lemmas}
\newtheorem{prop}[theoremcounter]{Proposition}
\crefname{prop}{Proposition}{Propositions}
\newtheorem{coro}[theoremcounter]{Corollary}
\crefname{coro}{Corollary}{Corollaries}

\crefname{conj}{Conjecture}{Conjectures}
\newtheorem{rmk}[theoremcounter]{Remark}
\crefname{rmk}{Remark}{Remarks}
\newtheorem{ex}[theoremcounter]{Example}
\crefname{ex}{Example}{Examples}

\crefname{fig}{Figure}{Figures} 
\crefformat{chap}{Chapter~#2#1#3}
\crefformat{sec}{Section~#2#1#3}
\crefformat{subsec}{Subsection#2#1#3}
\crefname{enumi}{}{}
\crefname{enumii}{}{}
\crefname{enumiii}{}{}
\crefname{enumiiii}{}{}
\newcommand{\HGF}[5]
	{\,_{#1}F_{#2}\left(\left.
		\begin{matrix}
			{#3}\\
			{#4}
		\end{matrix}\,\right|{#5}\right)
	}

\renewcommand{\d}{\mathrm{d}}

\newcommand{\CC}{\mathbb{C}}
\newcommand{\RR}{\mathbb{R}}
\newcommand{\QQ}{\mathbb{Q}}
\newcommand{\ZZ}{\mathbb{Z}}

\newcommand{\PP}{\mathbb{P}}

\begin{document}

	\title[]{Hurwitz–Lerch type central binomial series} 

	\author{Karin Ikeda} 
	\address{Joint Graduate School of Mathematics for Innovation, 
						Kyushu University, Motooka 744, Nishi-ku, Fukuoka 819-0395, Japan
					}
	\email{ikeda.karin.236@s.kyushu-u.ac.jp}

	\author{Yuta Kadono} 
	\address{Mathematical Institute, Tohoku University, 6-3, 
						Aramaki Aza-Aoba, Aoba-Ku, Sendai 980-8578, Japan.
					}
	\email{kadono.yuta.s6@dc.tohoku.ac.jp}

	\subjclass[2020]{11B68, 33C20}

	\maketitle

	\begin{abstract}
		The central binomial series is a subject that has been extensively studied, for example in the context of the 
		irrationality of Riemann zeta values. In this paper, the Hurwitz version of the central binomial series is defined by 
		adding one real parameter, and its values at integer points are studied.
	\end{abstract}

	\section{Introduction}

	The main an object of this paper is the \emph{Hurwitz type central binomial series} (HCBS) $\zeta_{HCB}(s,a)$, given as:
		\begin{align}
			\zeta_{HCB}(s,a)\coloneqq\sum_{n=0}^{\infty}\frac{1}{\binom{2(n+a)}{n+a}(n+a)^{s}}
		\end{align}
	for any complex numbers $s$ and real numbers $a\notin{\ZZ_{<1}}$. 
	The special case 
		\begin{align}
			\zeta_{CB}(s)\coloneqq\zeta_{HCB}(s,1)=\sum_{n=1}^{\infty}\frac{1}{\binom{2n}{n}n^{s}}
		\end{align}
	is classically known as the \emph{central binomial series} (CBS, see e.g., \cite{L,BBK}). 

	There have been many studies of the values of $\zeta_{CB}(s)$ at integer points $s=k\in{\ZZ}$, on which we recall the 
	following three results.
	First, Lehmer extended $\zeta_{CB}(1)$ to a one-variable function and found a connection with the arcsine function:
		\begin{thm}[Lehmer\cite{L}]\label{thm:Lehmer1}
			If $\abs{z}<1$, then
				\begin{align}
					\frac{2z\,\arcsin(z)}{\sqrt{1-z^{2}}}
						=\sum_{n=1}^{\infty}\frac{(2z)^{2n}}{\binom{2n}{n}n}. \label{eq:leh1}
				\end{align}
		\end{thm}

	Lehmer introduced the polynomials $p_{k}(x)$ and $q_{k}(x)$ (for $k\ge -1$) using the recursion
		\begin{align}
			p_{k+1}(x)&=2(kx+1)p_{k}(x)+2x(1-x)p_{k}'(x)+q_{k}(x),\\
			q_{k+1}(x)&=(2(k+1)x+1)q_{k}(x)+2x(1-x)q_{k}'(x),
		\end{align}
	with the initial values $p_{-1}(x)=0$ and $q_{-1}(x)=1$ to study the special values of CBS 
	(The first few polynomials are shown in \cref{table:ppaq} before \cref{thm:ptoE}).
	Then, through a successive differentiation, the following results for the values at negative integer points of the CBS 
	were obtained:
	\begin{thm}[Lehmer\cite{L}]\label{thm:Lehmer2} 
		For any integers $k\ge 0$, we have
			\begin{align}
				\sum_{n=1}^{\infty}\frac{(2n)^{k-1}(2z)^{2n}}{\binom{2n}{n}}
					=\frac{z}{(1-z^{2})^{k+\frac{1}{2}}}
						\left(z\sqrt{1-z^{2}}\ p_{k-1}(z^{2})+\arcsin(z)\,q_{k-1}(z^{2})\right).
			\end{align}
		In particular, when $z=1/2$, the following equality holds:
			\begin{align}
				\zeta_{CB}(1-k)
					=\left(\frac{2}{3}\right)^{k}
						\left(\frac{1}{2}\,p_{k-1}\left(\frac{1}{4}\right)+\frac{\pi}{3\sqrt{3}}\,q_{k-1}
						\left(\frac{1}{4}\right)\right)\in{\QQ+\QQ\frac{\pi}{\sqrt{3}}}.
			\end{align}
	\end{thm}

The following results are known for the polynomials $p_{k}(x)$ and $q_{k}(x)$:
    \begin{thm}[B\'{e}nyi-Matsusaka\cite{BM}]\label{th:BM}
        For any $n\ge 0$, we have
            \begin{align}
                p_{n}(x)&=2^{n}\sum_{k=0}^{n}\binom{n+1}{k}E_{n-k}\left(x,\frac{1}{2}\right)E_{k}\left(x,\frac{1}{2}\right),\\
                q_{n-1}(x)&=2^{n}E_{n}\left(x,\frac{1}{2}\right),\label{eq:qE}
            \end{align}
        where $E_{n}(x,y)$ is the bivariate Eulerian polynomial defined by
            \begin{align}
                \sum_{n=0}^{\infty}E_{n}(x,y)\frac{t^{n}}{n!}
                    =\left(\frac{1-x}{e^{t(x-1)}-x}\right)^{y}\eqqcolon\mathscr{E}(x,y;t).
            \end{align}
        Furthermore, the special value $p_{n}(1/4)$ is connected to the poly-Bernoulli numbers:
            \begin{align}
                \left(\frac{2}{3}\right)^{n}p_{n}\left(\frac{1}{4}\right)
                    &=\sum_{k=0}^{n}B_{n-k}^{(-k)},\label{eq:BM1}
            \end{align}
        where $B_{n}^{(k)}\in{\QQ}$ is the poly-Bernoulli number defined by
            \begin{align}
                \sum_{n=0}^{\infty}B_{n}^{(k)}\frac{t^{n}}{n!}
                    =\sum_{m=1}^{\infty}\frac{(1-e^{-t})^{m-1}}{m^{k}},
            \end{align}
        for any integers $k$.
    \end{thm}

    \begin{rmk}
        Note that bivariate Eulerian polynmomials $E_{n}(x,y)$ are classically defined as counting polynomials with respect to the permutations in the symmetric group $\mathfrak{S}_{n}$ (see e.g., B\'{e}nyi-Matsusaka\cite{BM}). Some examples are given below:
            \begin{align}
                E_{0}(x,y)&=1,\ \ E_{1}(x,y)=y,\ \ E_{2}(x,y)=y^{2}+xy,\\
                E_{3}(x,y)&=y^{3}+3xy+x^{2}y+xy.
            \end{align}
    \end{rmk}
    \begin{rmk}
        The identity \cref{eq:BM1} was observed experimentally by Stephan (see Kaneko\cite{K}).
        The first few examples of poly-Bernoulli numbers are given in the table below.
        \renewcommand{\arraystretch}{1.2}
            \begin{table}[H]
                \begin{center}
                    \begin{tabular}{|c||c|c|c|c|c|}\hline
                         $k\setminus n$ & $0$ & $1$ &  $2$ &  $3$  &  $4$  \\ \hhline{|=#=|=|=|=|=|}
                               $0$      & $1$ & $1$ & $1$ & $1$ & $1$  \\ \hline
                               $1$      & $1$ & $2$ & $4$ & $8$ & $16$  \\ \hline
                               $2$      & $1$ & $4$ & $14$ & $46$ & $146$ \\ \hline
                               $3$      & $1$ & $8$ & $46$ & $230$ & $1066$ \\ \hline
                               $4$      & $1$ & $16$ & $146$ & $1066$ & $6902$ \\ \hline
                    \end{tabular}
                    \caption{Poly-Bernoulli numbers $B_{n}^{(-k)}$.}
                \end{center}
            \end{table}
        For example, we compute from this table the right-hand side of \cref{eq:BM1} for $n=3$ as 
        \begin{align}
         B_{3}^{(0)}+B_{2}^{(-1)}+B_{1}^{(-2)}+B_{0}^{(-3)}=1+4+4+1=10.
        \end{align}
        On the other hand, we have $p_3(x)=20x^2+70x+15$ (\cref{table:ppaq}) and so
        \begin{align}
            \left(\frac{2}{3}\right)^3p_{3}\left(\frac{1}{4}\right)=\frac{8}{27}\left(\frac{5}{4}+\frac{35}{2}+15\right)=10,
        \end{align}
        thus \cref{eq:BM1} is verified in this case.
    \end{rmk}
On the other hand, Borwein, Broadhurst, and Kamnitzer showed that $\zeta_{CB}(k)$ for integers $k\ge2$ can be written as a rational linear combination of multiple zeta, Clausen and Glaisher values through log-sine integral (see\cite{BBK}).

The current paper studies the values of $\zeta_{HCB}(k,a)$ at integer points $s=k\in{\ZZ_{\le 1}}$.
For this purpose, \cref{sec:Def} introduces the \emph{Hurwitz--Lerch type central binomial series} (HLCBS):
    \begin{align}
        \Phi_{HCB}(s,a,z)\coloneqq\sum_{n=0}^{\infty}\frac{(2z)^{2(n+a)}}{\binom{2(n+a)}{n+a}(n+a)^{s}},
    \end{align} 
which extends HCBS by adding one more variable $z$, and gives a representation of it when $s$ is an integer using generalised hypergeometric functions (see \cref{choukika1}).
In \cref{sec:main}, by applying Euler's transformation formula to this hypergeometric series representation, we provide a generalisation of \cref{thm:Lehmer1} (see \cref{choukika2}).
Furthermore, results at negative integer points of the HCBS were obtained by using the same method as in \cref{thm:Lehmer2} (see \cref{zenka}).

\section{Definitions and Preliminaries}\label{sec:Def}

First, we introduce HLCBS and HCBS the Hurwitz type generalizations of the classical CBS, which are the subject of this paper.
    \begin{dfn}[Hurwitz--Lerch type central binomial series (HLCBS)]\label{def:HLCBS}
        For any $|z|<1$, $a\in{\RR\setminus\ZZ_{\leq 0}}$ and $s\in\CC$, we define $\Phi_{HCB}(s,a,z)$ by this series
            \begin{align}
                \Phi_{HCB}(s,a,z)\coloneqq\sum_{n=0}^{\infty}\frac{(2z)^{2(n+a)}}{\binom{2(n+a)}{n+a}(n+a)^s}.
            \end{align} 
        Here, the binomial coefficient is extended to the real arguments $x$ and $y$ as follows:
            \begin{align}
                \binom{x}{y}\coloneqq\frac{\Gamma(x+1)}{\Gamma(y+1)\Gamma(x-y+1)}.
            \end{align}
    \end{dfn}

We note that, for any positive integers $m$, we have 
    \begin{align}
			\Phi_{HCB}\left(s,-m+\frac{1}{2},z\right)
				&=\sum_{n=-m}^{\infty}
					\frac{(2z)^{2\left(n+\frac{1}{2}\right)}}
					{\binom{2(n+\frac{1}{2})}{n+\frac{1}{2}}\left(n+\frac{1}{2}\right)^s}
				=\sum_{n=0}^{\infty}
					\frac{(2z)^{2\left(n+\frac{1}{2}\right)}}
					{\binom{2(n+\frac{1}{2})}{n+\frac{1}{2}}\left(n+\frac{1}{2}\right)^s}\\
				&=\Phi_{HCB}\left(s,\frac{1}{2},z\right)
    \end{align}
because $\binom{2\left(n+1/2\right)}{n+1/2}^{-1}=0$ for negative $n$.
Therefore, $a$ is henceforth assumed to be a real number that is not a half-integer less than or equal to zero.
We are interested in assigning various values to the third variable $z$ in the HLCBS. 
Among them, we will refer to the one substituting $z=1/2$ as \emph{Hurwitz type central binomial series} (HCBS) and denote it as 
    \begin{align}
        \zeta_{HCB}(s,a)\coloneqq\Phi_{HCB}\left(s,a,\frac{1}{2}\right)
            =\sum_{n=0}^{\infty}\frac{1}{\binom{2(n+a)}{n+a}(n+a)^s}.
    \end{align}
It is clear from the definition that HCBS is a generalisation of CBS, since $\zeta_{HCB}(s,1)=\zeta_{CB}(s)$. 
In particular, this paper focuses on the special values of the HLCBS (or HCBS) when the first variable $s$ is an integer.

By definition, the following relation and hypergeometric series representations can be easily obtained by direct calculation. 

    \begin{lem}[Defferential relation]\label{lem:df}
        For any $|z|<1$, $a\in\mathbb{R}\backslash\frac{1}{2}\mathbb{Z}_{\leq 0}$ and $s\in\mathbb{C}$, we have
        \begin{align}
            \frac{1}{2}\theta_z\Phi_{HCB}(s,a,z)=\Phi_{HCB}(s-1,a,z).\label{eq:lem1}
        \end{align}
        Here, $\theta_{z}\coloneqq z\frac{\d}{\d z}$ is the Euler operator.
    \end{lem}

    \begin{prop}\label{choukika1}
        For any $|z|<1$, $a\in\mathbb{R}\backslash\frac{1}{2}\mathbb{Z}_{\leq 0}$ and $k\in{\ZZ_{> 0}}$, we have
            \begin{align}
                \Phi_{HCB}(k,a,z)&=\frac{4^a}{\binom{2a}{a}a^k}z^{2a}
                    \HGF{k+1}
                        {k}
                        {1,\,a,\,\ldots,\,a}
                        {a+\frac{1}{2},\,a+1,\,\ldots,\,a+1}
                        {z^2},\label{eq:prop1-1}\\
                \Phi_{HCB}(1-k,a,z)&=\frac{4^a}{\binom{2a}{a}a^{1-k}}z^{2a}
                    \HGF{k+1}
                        {k}
                        {1,\,a+1,\,\ldots,\,a+1}
                        {a+\frac{1}{2},\,a,\,\ldots,\,a}
                        {z^2}.\label{eq:prop1-2}
            \end{align}
        Here, the generalized hypergeometric series is defined by
        \begin{align}
            \HGF{p+1}
                {p}
                {\alpha_{0},\,\alpha_{1},\,\ldots,\,\alpha_{p}}
                {\beta_{1},\,\ldots,\,\beta_{p}}
                {z}
            \coloneqq\sum_{n=0}^{\infty}
                \frac{(\alpha_{0})_{n}(\alpha_{1})_{n}\cdots(\alpha_{p})_{n}}  {(\beta_{1})_{n}\cdots(\beta_{p})_{n}(1)_{n}}z^{n}, 
        \end{align}
        for $\alpha_{0}$, $\alpha_{1}$, $\ldots$, $\alpha_{p}\in{\CC}$, 
        $\beta_{1}$, $\ldots$, $\beta_{p}\in{\CC\setminus\ZZ_{\leq 0}}$, 
        and shifted factorial $(\alpha)_{n}$ is defined as 
        \begin{align}
            (\alpha)_n:=\frac{\Gamma(\alpha+n)}{\Gamma(\alpha)}=\left\{
            \begin{array}{ll}
                \displaystyle\prod_{l=0}^{n-1}(\alpha+l) & n\neq 0,\\
                \\
                1 & n=0,
            \end{array}\right.
        \end{align}
        for any $\alpha\in{\CC}$ and $n\in{\ZZ_{\ge 0}}$.
    \end{prop}

Through the differential relation \cref{eq:lem1}, the following relationship between \cref{eq:prop1-1,eq:prop1-2} holds
    \begin{align}
        \begin{array}{c}
            \vdots\\
            \downarrow\\
            \Phi_{HCB}(2,a,z)=\dfrac{4^a}{\binom{2a}{a}a^{2}}z^{2a}
                \HGF{3}
                    {2}
                    {1,\,a,\,a}
                    {a+\frac{1}{2},\,a+1}
                    {z^2}\\
            \downarrow\\
            \Phi_{HCB}(1,a,z)=\dfrac{4^a}{\binom{2a}{a}a}z^{2a}
                \HGF{2}
                    {1}
                    {1,\,a}
                    {a+\frac{1}{2}}
                    {z^2}\\
            \downarrow\\
              \Phi_{HCB}(0,a,z)=\dfrac{4^a}{\binom{2a}{a}}z^{2a}
                \HGF{2}
                    {1}
                    {1,\,a+1}
                    {a+\frac{1}{2}}
                    {z^2}\\
            \downarrow\\
              \Phi_{HCB}(-1,a,z)=\dfrac{4^{a}a}{\binom{2a}{a}}z^{2a}
                \HGF{3}
                    {2}
                    {1,\,a+1,\,a+1}
                    {a+\frac{1}{2},\,a}
                    {z^2}\\
            \downarrow\\
            \vdots\\
        \end{array}
    \end{align}
where $\downarrow$ means that $\frac{1}{2}\theta_{z}$ is acted on each side. 
We find it amusing that both $\Phi_{HCB}(k,a,z)$ and $\Phi_{HCB}(1-k,a,z)$ have the same $_{k+1}F_{k}$ and the symmetry $a\longleftrightarrow a+1$ in the parameters in $_{k+1}F_{k}$ except for common $1$ and $a+1/2$.

In the next section, we give an alternative hypergeometric expression of $\Phi_{HCB}(1,a,z)$, which allows us to generalise Lehmer's \cref{thm:Lehmer2}. We observe that from \cref{lem:df} the case of $s=1$ $(\Phi_{HCB}(1,a,z))$ is more or less essential.

\section{Special values of HLCBS and HCBS}\label{sec:main}
We start with the following theorem.
    \begin{thm}\label{choukika2}
        For any $|z|<1$ and $a\in\mathbb{R}\backslash\frac{1}{2}\mathbb{Z}_{\leq 0}$, we have
            \begin{align}\label{eq:thm1}
                \Phi_{HCB}(1,a,z)
                =\frac{4^a}{\binom{2a}{a}a}\times\frac{z^{2a}}{\sqrt{1-z^2}}
                    \HGF{2}
                        {1}
                        {\frac{1}{2}, a-\frac{1}{2}}
                        {a+\frac{1}{2}}
                        {z^2}.
            \end{align}
    \end{thm}

    \begin{proof}
        The claim of \cref{choukika2} is obtained immediately by applying Euler's transformation formula to the Gaussian hypergeometric function on the right-hand side of \cref{eq:prop1-1} for $k=1$: 
            \begin{align}\label{eq:pfthm1-1}
                \Phi_{HCB}(1,a,z)
                =\frac{4^a}{\binom{2a}{a}a}z^{2a}
                    \HGF{2}
                        {1}
                        {1,\,a}
                        {a+\frac{1}{2}}
                        {z^2}.
            \end{align}
        However, a more elementary proof is given here. 
        If $a=1/2$, the claim follows from equation \cref{eq:pfthm1-1} because
            \begin{align}
                \HGF{2}{1}{\frac{1}{2}, 0}{1}{z^2}=1\quad\text{and}\quad\HGF{2}{1}{1, \frac{1}{2}}{1}{z^2}=\frac{1}{\sqrt{1-z^2}}.
            \end{align}
        So in the following we assume $a\neq 1/2$.
        The difference of the right-hand sides of \cref{eq:thm1,eq:pfthm1-1} is 
            \begin{align}
            &\frac{4^a}{\binom{2a}{a}a}\frac{z^{2a}}{\sqrt{1-z^2}}\left(\sqrt{1-z^2}\HGF{2}{1}{1, a}{a+\frac{1}{2}}{z^2}-\HGF{2}{1}{\frac{1}{2}, a-\frac{1}{2}}{a+\frac{1}{2}}{z^2}\right)\\
            &=\frac{4^a}{\binom{2a}{a}a}\frac{z^{2a}}{\sqrt{1-z^2}}\sum_{n=0}^{\infty}\left(\sum_{m=0}^{n}\frac{\left(-\frac{1}{2}\right)_{m}(a)_{n-m}}{\left(a+\frac{1}{2}\right)_{n-m}m!}-\frac{\left(\frac{1}{2}\right)_n\left(a-\frac{1}{2}\right)_n}{\left(a+\frac{1}{2}\right)_n n!}\right)z^{2n}\\
                &=\frac{4^a}{\binom{2a}{a}a}
                \frac{z^{2a}}{\sqrt{1-z^2}}
                \sum_{n=0}^{\infty}
                    \frac{(a)_n}{\left(a+\frac{1}{2}\right)_n}
                    \left(
                        \sum_{m=0}^{n}
                            \frac{\left(-\frac{1}{2}\right)_m\left(\frac{1}{2}-a-n\right)_m}{(1-a-n)_m m!}
                        -\frac{\left(\frac{1}{2}\right)_n\left(a-\frac{1}{2}\right)_n}{\left(a+\frac{1}{2}\right)_{n} n!}
                    \right)z^{2n}.
            \end{align}
        Here, it can be seen that the coefficient 
            \begin{align}\label{kinou}
                \sum_{m=0}^{n}
                    \frac{\left(-\frac{1}{2}\right)_m\left(\frac{1}{2}-a-n\right)_m}{(1-a-n)_m m!}
                -\frac{\left(\frac{1}{2}\right)_n\left(a-\frac{1}{2}\right)_n}{\left(a+\frac{1}{2}\right)_n n!}
            \end{align}
        of the series vanishes for any integers $n\ge 0$. In fact, it follows immediately when $n=0$ and $1$. If \cref{kinou} vanishes up to $k=n-1$, then when $k=n$ we have
            \begin{align}
                \sum_{m=0}^{n}
                    \frac{\left(-\frac{1}{2}\right)_m\left(\frac{1}{2}-a-n\right)_{m}}{\left(1-a-n\right)_{m}m!}
                &=\frac{\left(-\frac{1}{2}\right)_n\left(\frac{1}{2}-a-n\right)_{n}}{\left(1-a-n\right)_{n}n!}
                +\sum_{m=0}^{n-1}
                    \frac{\left(-\frac{1}{2}\right)_m\left(\frac{1}{2}-(a+1)-(n-1)\right)_{m}}{\left(1-(a+1)-(n-1)\right)_{m}m!}\\
                &=\frac{\left(\frac{1}{2}\right)_n\left(a-\frac{1}{2}\right)_{n}}{\left(a\right)_{n}n!}
                \left(
                    \frac{2a+2n-1}{(1-2n)(2a-1)}+\frac{4na}{(2n-1)(2a-1)}
                \right)\\
                &=\frac{\left(\frac{1}{2}\right)_n\left(a-\frac{1}{2}\right)_{n}}{\left(a\right)_{n}n!}.
            \end{align}
        By induction, \cref{kinou} vanishes for any integers $n\ge 0$.
    \end{proof}

    \begin{rmk}
         We note that, from \cref{eq:pfthm1-1}, $\Phi_{HCB}(1,a,z)$ satisfies the following {\em first-order} differential equation:
            \begin{align}
                \left[(1-z^{2})z\frac{\d}{\d z}-1\right]y=(2a-1)z^{2a},
            \end{align}
        which is a Fuchsian type with regular singularities at $z=0$, $1$, $-1$ and $\infty$ in $\PP^{1}(\CC)\coloneqq\CC\cup\{\infty\}$. In particular, when $a\neq1/2$, the right-hand side of \cref{eq:thm1} is obtained as a general solution of this equation.
    \end{rmk}

We note that \cref{choukika2} is a generalisation of \cref{thm:Lehmer1}. In fact,
    \begin{align}
        \Phi_{HCB}(1,1,z)=\frac{2z^2}{\sqrt{1-z^2}}\HGF{2}{1}{\frac{1}{2}, \frac{1}{2}}{\frac{3}{2}}{z^2}=\frac{2z\,\arcsin(z)}{\sqrt{1-z^{2}}}.
    \end{align}
Furthermore, as will be discussed later, it can be seen that the derivative of the right-hand side is an analogue of the $\arcsin(z)$ derivative when considered (see \cref{eq:dHGF}).
Consider the Euler operator acting on \cref{eq:thm1} $k$-times.
Then, from \cref{lem:df}, we have the following theorem:
    \begin{thm}\label{zenka}
        For any $a\in\mathbb{R}$, we define the sequence of polynomials $\bigl(p_k(a, x)\bigr)_{k\geq -1}$ by the recursion
            \begin{align}
                p_{k+1}(a, x)&=2((k+1-a)x+a)p_{k}(a, x)+2x(1-x)p_{k}'(a, x)+q_{k}(x)   
            \end{align}
        with initial value $p_{-1}(a,x)=0$ ($q_k(x)$ is the previously defined polynomial).
        Then, for any $k\in{\ZZ_{\ge 0}}$ and $a\in{\RR\setminus\frac{1}{2}\ZZ_{\leq 0}}$, we have
            \begin{align}
                &2^{k-1}\Phi_{HCB}(1-k,a,z)\\
                &\quad=\frac{4^a z^{2a}}{2a\binom{2a}{a}(1-z^2)^{k+\frac{1}{2}}}
                    \left(
                        (2a-1)\sqrt{1-z^2}p_{k-1}(a,z^2)+
                        \HGF{2}
                            {1}
                            {\frac{1}{2}, a-\frac{1}{2}}
                            {a+\frac{1}{2}}
                            {z^2}
                            q_{k-1}(z^2)
                    \right).\label{eq:zenka1}
            \end{align}
    \end{thm}

    \begin{proof}
        Denote 
            \begin{align}
                u_1(z)\coloneqq z^{2a-1}\HGF{2}{1}{\frac{1}{2}, a-\frac{1}{2}}{a+\frac{1}{2}}{z^2},
                \qquad u_2(z)\coloneqq z^{2a-1}\sqrt{1-z^2}.
            \end{align}
        We show the claim by induction on $k$, using \cref{lem:df}. The case of $k=0$ is trivially holds. Assume the claim is true for some arbitrary $k=n\ge 0$. 
        We note that, for any $a\in{\RR\setminus\frac{1}{2}\ZZ_{\le 0}}$, we obtain
            \begin{align}
                \frac{\d}{\d z}\left(z^{2a-1}\HGF{2}{1}{\frac{1}{2},a-\frac{1}{2}}{a+\frac{1}{2}}{z^2}\right)=(2a-1)\frac{z^{2a-2}}{\sqrt{1-z^2}}.\label{eq:dHGF}
            \end{align}
        Therefore by \cref{lem:df,eq:dHGF}, we have        
        \begin{align}
            &2^{n}\Phi_{HCB}(1-(n+1),a,z)=\theta_{z}(2^{n-1}\Phi_{HCB}(1-n,a,z))\\
            &\quad=\frac{4^{a}z}{2a\binom{2a}{a}}\frac{\d}{\d z}
                \left(
                    \frac{z}{(1-z^2)^{n+\frac{1}{2}}}u_{1}(z)q_{n-1}(z^2)
                    +(2a-1)\frac{z}{(1-z^2)^{n+\frac{1}{2}}}u_{2}(z)p_{n-1}(a,z^2)
                \right)\\
            &\quad=\frac{4^{a}z}{2a\binom{2a}{a}(1-z^2)^{n+1+\frac{1}{2}}}
                \Bigg(
                    \left(
                        (2nz^2+1)q_{n-1}(z^2)+2z^2(1-z^2)q_{n-1}'(z^2)
                    \right)u_1(z)\\
                    &\qquad+(2a-1)
                        \Bigl(
                            2\bigl((n-a)z^{2}+a\bigr)p_{n-1}(a,z^{2})
                            +2z^2(1-z^2)p_{n-1}'(a,z^2)
                            +q_{n-1}(z^2)
                        \Bigr)u_2(z)
                \Bigg)\\
            &\quad=\frac{4^az}{2a\binom{2a}{a}(1-z^2)^{n+1+\frac{1}{2}}}
                \left(q_{n}(z^2)u_1(z)+(2a-1)p_{n}(a,z^2)u_2(z)\right).
        \end{align}
        Thus, the claim holds for any integers $k\ge 0$, and the proof is complete.
    \end{proof}

This theorem is a generalisation involving \cref{thm:Lehmer2}. 
    Furthermore, the polynomial $q_k(x)$ is the same as that appearing in \cref{thm:Lehmer2}, so it is related to the bivariable Eulerian polynomial as in \cref{th:BM}. 
On the other hand, the polynomial $p_{k}(a,x)$ interpolates $p_{k}(x)$ and $q_{k}(x)$, that is $p_{k}(0,x)=q_{k}(x)$ (for $k>-1$), and $p_{k}(1,x)=p_{k}(x)$ (for $k\ge -1$). 
    \begin{table}[H]
        \begin{center}
            \begin{tabular}{|c|c|c|c|} \hline
               $ n$ &    $p_{n}(x)$    &                 $p_{n}(a,x)$                  & $q_{n}(x)$ \\ \hline
               $-1$ &        $0$       &                     $0$                       &    $1$     \\ \hline
               $0$  &        $1$       &                     $1$                       &    $1$     \\ \hline
               $1$  &        $3$       &               $2(1-a)x+2a+1$                  &    $2x+1$     \\ \hline
               $2$  &      $8x+7$      & \begin{tabular}{c}
																					$4(1-a)^{2}x^{2}-2(4a^{2}-3a-5)x$\\
																					$+4a^{2}+2a+1$
																				\end{tabular} &    $4x^{2}+10x+1$     \\ \hline
               $3$  & $20x^{2}+70x+15$ & \begin{tabular}{c}
                                            $8(1-a)^{3}x^{3}$\\
																						$+4(6a^{3}-11a^{2}-5a+15)x^{2}$\\
                                            $-2(12a^{3}-8a^{2}-21a-18)x$\\
																						$+(2a+1)(4a^{2}+1)$
                                         \end{tabular} &    $8x^{3}+60x^{2}+36x+1$     \\ \hline
            \end{tabular}
            \caption{First small examples of $p_{n}(x)$, $p_{n}(a,x)$, and $q_{n}(x)$.}
            \label{table:ppaq}
        \end{center}
    \end{table}
    
Besides, an Eulerian polynomial representation can also be given for $p_{n}(a,x)$, as follows.
    \begin{thm}\label{thm:ptoE}
        For any $n\in{\ZZ_{\ge 0}}$, $a\in{\RR\setminus\frac{1}{2}\ZZ_{\le 1}}$ and $|z|<1$, we have 
            \begin{align}
                p_{n}(a,z)=2^{n}\sum_{j=0}^{n}\sum_{l=0}^{j}\binom{n+1}{j+1}\binom{j}{l}
                    (a-1)^{j-l}(1-z)^{j-l}E_{n-j}\left(z,\frac{1}{2}\right)
                    E_{l}\left(z,\frac{1}{2}\right).
            \end{align} 
    \end{thm}
    \begin{proof}
	From \cref{eq:zenka1}, we have
		\begin{align}
			&\sum_{k=0}^{\infty}\sum_{n=1}^{\infty}\frac{(2(n+a))^{k-1}(2z)^{2(n+a)}}{\binom{2(n+a)}{n+a}}
				\frac{t^{k}}{k!}\\
			&=\frac{4^a z^{2a}}{2a\binom{2a}{a}\sqrt{1-z^2}}
				\Biggl(
					(2a-1)P\left(a,z^{2};\frac{t}{1-z^{2}}\right)\sqrt{1-z^2}\\
			&\quad\quad+\HGF{2}
				{1}
				{\frac{1}{2}, a-\frac{1}{2}}
				{a+\frac{1}{2}}
				{z^2}
				Q\left(z^2;\frac{t}{1-z^{2}}\right)
				\Biggr),
		\end{align}
        where the generating functions
            \begin{align}
                P(a,x;t)&\coloneqq\sum_{n=0}^{\infty}p_{n-1}(a,x)\frac{t^{n}}{n!},\ \ 
                Q(x;t)\coloneqq\sum_{n=0}^{\infty}q_{n-1}(x)\frac{t^{n}}{n!}.
            \end{align}
        By applying \cref{eq:thm1}, we have
						\begin{align}
							\sum_{k=0}^{\infty}\sum_{n=1}^{\infty}
								\frac{(2(n+a))^{k-1}(2z)^{2(n+a)}}{\binom{2(n+a)}{n+a}}
								\frac{t^{k}}{k!}
							&=\sum_{n=1}^{\infty}\frac{(2ze^{t})^{2(n+a)}}{\binom{2(n+a)}{n+a}(2(n+a))}\\
							&=\frac{4^{a}}{2a\binom{2a}{a}}\frac{e^{2at}z^{2a}}{\sqrt{1-e^{2t}z^{2}}}
								\HGF{2}
										{1}
										{\frac{1}{2}, a-\frac{1}{2}}
										{a+\frac{1}{2}}
										{e^{2t}z^2}.
						\end{align}
				Since \cref{eq:qE} holds that $Q(x;t)=\mathscr{E}\left(x,1/2;2t\right)$, we have
            \begin{align}
                &P\left(a,z^{2};\frac{t}{1-z^{2}}\right)\\
                &=\frac{e^{2at}}{(2a-1)\sqrt{1-e^{2t}z^{2}}}
                    \HGF{2}
                        {1}
                        {\frac{1}{2}, a-\frac{1}{2}}
                        {a+\frac{1}{2}}
                        {e^{2t}z^2}\\
                &\ \ \ \ \ -\mathscr{E}\left(z^2,\frac{1}{2};\frac{2t}{1-z^2}\right)\frac{1}{(2a-1)\sqrt{1-z^{2}}}
                    \HGF{2}
                        {1}
                        {\frac{1}{2}, a-\frac{1}{2}}
                        {a+\frac{1}{2}}
                        {z^2}\\
                &=\frac{e^{t}}{(2a-1)\sqrt{1-e^{2t}z^{2}}}\left(e^{(2a-1)t}
                    \HGF{2}
                        {1}
                        {\frac{1}{2}, a-\frac{1}{2}}
                        {a+\frac{1}{2}}
                        {e^{2t}z^2}-
                    \HGF{2}
                        {1}
                        {\frac{1}{2}, a-\frac{1}{2}}
                        {a+\frac{1}{2}}
                        {z^2}
                    \right).
            \end{align}
	From this, we obtain the following:
		\begin{align}
			P(a,z;t)&=\frac{e^{(1-z)t}}{(2a-1)\sqrt{1-e^{2(1-z)t}z}}\\
				&\quad\quad\times\left(e^{(2a-1)(1-z)t}
					\HGF{2}
							{1}
							{\frac{1}{2}, a-\frac{1}{2}}
							{a+\frac{1}{2}}
							{e^{2(1-z)t}z}
				 -\HGF{2}
							{1}
							{\frac{1}{2}, a-\frac{1}{2}}
							{a+\frac{1}{2}}
							{z}
					\right)\\
				&=\mathscr{E}\left(z,\frac{1}{2};2t\right)\frac{1}{(2a-1)\sqrt{1-z}}\\
				&\quad\quad\times\left(e^{(2a-1)(1-z)t}
					\HGF{2}
							{1}
							{\frac{1}{2}, a-\frac{1}{2}}
							{a+\frac{1}{2}}
							{e^{2(1-z)t}z}
				 -\HGF{2}
							{1}
							{\frac{1}{2}, a-\frac{1}{2}}
							{a+\frac{1}{2}}
							{z}
			\right).
		\end{align}
        Since
            \begin{align}
                &\frac{\d}{\d t}\left(e^{(2a-1)(1-z)t}
                    \HGF{2}
                        {1}
                        {\frac{1}{2}, a-\frac{1}{2}}
                        {a+\frac{1}{2}}
                        {e^{2(1-z)t}z}-
                    \HGF{2}
                        {1}
                        {\frac{1}{2}, a-\frac{1}{2}}
                        {a+\frac{1}{2}}
                        {z}
                    \right)\\
                &\quad=(2a-1)(1-z)e^{(2a-1)(1-z)t}\sum_{n=0}^{\infty}\left(\frac{1}{2}\right)_{n}\frac{(e^{2(1-z)t}z)^{n}}{n!}\\
                &\quad=(2a-1)\sqrt{1-z}\ e^{2(a-1)(1-z)t}\mathscr{E}\left(z,\frac{1}{2};2t\right),
            \end{align}
        it holds that
            \begin{align}
                &\frac{1}{(2a-1)\sqrt{1-z}}\left(e^{(2a-1)(1-z)t}
                    \HGF{2}
                        {1}
                        {\frac{1}{2}, a-\frac{1}{2}}
                        {a+\frac{1}{2}}
                        {e^{2(1-z)t}z}-
                    \HGF{2}
                        {1}
                        {\frac{1}{2}, a-\frac{1}{2}}
                        {a+\frac{1}{2}}
                        {z}
                    \right)\\
                &\quad=\sum_{n=0}^{\infty}\left(2^{n}\sum_{l=0}^{n}\binom{n}{l}(a-1)^{n-l}(1-z)^{n-l}E_{l}\left(z,\frac{1}{2}\right)\right)\frac{t^{n+1}}{(n+1)!}.
            \end{align}
	Thus, we have
		\begin{align}
			&P(a,z;t)\\
			&\quad=\sum_{n=0}^{\infty}\left(2^{n}\sum_{j=0}^{n}\sum_{l=0}^{j}\binom{n+1}{l+1}\binom{j}{l}
				(a-1)^{j-l}(1-z)^{j-l}E_{n-j}\left(z,\frac{1}{2}\right)E_{l}\left(z,\frac{1}{2}\right)\right)\frac{t^{n+1}}{(n+1)!},
		\end{align}
	which concludes the proof.
	\end{proof}

    \begin{rmk}
        From the calculation of the generating function used in the proof of \cref{thm:ptoE}, it can be shown that $p_{n}(0,x)=q_{n}(x)$ (for $n>-1$). In fact, the following equality holds:
            \begin{align}
                P(0,z;t)&=\mathscr{E}\left(z,\frac{1}{2};2t\right)\frac{-1}{\sqrt{1-z}}\left(e^{(z-1)t}
                    \HGF{1}
                        {0}
                        {-\frac{1}{2}}
                        {-}
                        {e^{2(1-z)t}z}-
                    \HGF{1}
                        {0}
                        {-\frac{1}{2}}
                        {-}
                        {z}
                    \right)\\
                &=\mathscr{E}\left(z,\frac{1}{2};2t\right)\frac{-1}{\sqrt{1-z}}\left(\sqrt{e^{2(z-1)t}-z}-\sqrt{1-z}\right)\\
                &=\mathscr{E}\left(z,\frac{1}{2};2t\right)-1\\
                &=Q(z;t)-1.
            \end{align}
    \end{rmk}

    We have not yet found anything like poly-Bernoulli numbers that can be associated with $p_{n}(a,x)$. Here is a fact known about $p_{n}(a,1/4)$.
    
    \begin{prop}
        The sequence $(\alpha_{n}(a))_{n\ge 0}$ defined as
            \begin{align}
                \alpha_{n}(a)\coloneqq\left(\frac{2}{3}\right)^{n}p_{n}\left(a,\frac{1}{4}\right)
            \end{align}
        satisfies the recursion
            \begin{align}
                    3\alpha_{n}(a)=2\alpha_{n-1}(a)+\sum_{l=0}^{n-1}\binom{n}{l}\alpha_{l}(a)+3a^{n},
            \end{align}
        with the initial value $\alpha_{0}(a)=1$.
    \end{prop}

    \begin{proof}
        The generating function for $\alpha_{n}(a)$ is given by 
            \begin{align}
                \sum_{n=-1}^{\infty}\alpha_{n}(a)\frac{t^{n+1}}{(n+1)!}
                &=\frac{3}{2}P\left(a,\frac{1}{4};\frac{2}{3}t\right)\\
                &=\frac{3\,e^{\frac{t}{2}}}{(2a-1)\sqrt{4-e^{t}}}\left(e^{\frac{2a-1}{2}t}
                        \HGF{2}
                            {1}
                            {\frac{1}{2}, a-\frac{1}{2}}
                            {a+\frac{1}{2}}
                            {\frac{e^{t}}{4}}-
                        \HGF{2}
                            {1}
                            {\frac{1}{2}, a-\frac{1}{2}}
                            {a+\frac{1}{2}}
                            {\frac{1}{4}}
                        \right).
            \end{align}
        Since the right-hand side satisfies the differential equation
            \begin{align}
                \left[(4-e^{t})\frac{\d}{\d t}-2\right]y=3e^{at},
            \end{align}
        the coefficients $\alpha_{n}(a)$ satisfy the desired recurrence formula.
    \end{proof}

From  \cref{zenka}, the special values of HCBS are described based on the results obtained for HLCBS.

    \begin{thm}\label{zetatokushu}
        For any $k\in{\ZZ_{\geq 0}}$ and  $a\in{\RR_{> \frac{1}{2}}}$, we have 
        \begin{align}
            \zeta_{HCB}(1-k,a)\in{\frac{\Gamma(a+1)^2}{\Gamma(2a+1)}}
                \left(
                    \QQ(a)+
                    \frac{4^{a}}{\sqrt{3}}
                        B\left(\frac{1}{4};a-\frac{1}{2},\frac{1}{2}\right)\QQ(a)
                \right),
        \end{align}
        where  $B(z;\alpha,\beta)$ is the incomplete beta function:
        \begin{align}
            B(z;\alpha,\beta):=\int_{0}^{z}x^{\alpha-1}(1-x)^{\beta-1}\d x,
        \end{align}
        for $\alpha$, $\beta$, $z\in{\CC}$ with $\mathrm{Re}(\alpha)$, 
        $\mathrm{Re}(\beta)>0$, and $0\le\mathrm{Re}(z)\le1$.
    \end{thm}

	\begin{proof}
		From \cref{zenka}, if $z=1/2$, we obtain 
			\begin{align}
				&\zeta_{HCB}(1-k,a)\\
				&\quad=\frac{\Gamma(a+1)^2}{a\Gamma(2a+1)}
					\left(\frac{2}{3}\right)^{k}
					\left(
						(2a-1)p_{k-1}\left(a,\frac{1}{4}\right)
							+\frac{2}{\sqrt{3}}
							\HGF{2}
									{1}
									{\frac{1}{2}, a-\frac{1}{2}}
									{a+\frac{1}{2}}
									{\frac{1}{4}}
							q_{k-1}\left(\frac{1}{4}\right)
					\right).
			\end{align}
		In particular, since
			\begin{align}
				\HGF{2}
						{1}
						{\frac{1}{2}, a-\frac{1}{2}}
						{a+\frac{1}{2}}
						{\frac{1}{4}}
				&=\frac{\Gamma(a+\frac{1}{2})}{\Gamma(a-\frac{1}{2})\Gamma(1)}
					\int_{0}^{1}x^{a-\frac{3}{2}}
					\left(1-\frac{x}{4}\right)^{-\frac{1}{2}}\d x\\
				&=4^{a-1}(2a-1)B\left(\frac{1}{4}; a-\frac{1}{2}, \frac{1}{2}\right),
			\end{align}
		we have
			\begin{align}
				&\zeta_{HCB}(1-k,a)\\
				&\quad=\frac{(2a-1)\Gamma(a+1)^2}{a\Gamma(2a+1)}\left(\frac{2}{3}\right)^{k}
					\left(
						p_{k-1}\left(a,\frac{1}{4}\right)
						+4^{a-1}\frac{2}{\sqrt{3}}B\left(\frac{1}{4}; a-\frac{1}{2}, \frac{1}{2}\right)q_{k-1}\left(\frac{1}{4}\right)
					\right).
			\end{align}
		Since $p_{k}(a,x)$ and $q_{k}(x)$ are polynomials of rational coefficients with respect to $x$, we obtain the claim.
	\end{proof}

We show the results for some special $a$. 
If $a$ is an integer $m>0$, then 
        \begin{align}
            \zeta_{HCB}(1-k,m)\in{\QQ+\QQ\frac{\pi}{\sqrt{3}}}
        \end{align}
follows. This is a direct consequence of the following equality and \cref{thm:Lehmer2}:
    \begin{align}
        \zeta_{HCB}(1-k,a+1)
        =\zeta_{HCB}(1-k,a)-\frac{1}{\binom{2a}{a}a^{1-k}}.
    \end{align}
If $a$ is a positive half-integer greater than 1/2, then we have the following:
    \begin{coro}
        For any integers $k\geq0$ and $m\geq 1$, $\zeta_{HCB}(1-k,m+1/2)$ is in a two-dimensional $\QQ$-vector space spanned by $\pi$ and $\pi/\sqrt{3}$, i.e., 
            \begin{align}
                \zeta_{HCB}\left(1-k,m+\frac{1}{2}\right)\in{\QQ\bigl(\sqrt{3}\bigr)\pi}.
            \end{align}
    \end{coro}

    \begin{proof}
        This can be shown by using \cref{zetatokushu} and the following relations
            \begin{align}
                B(x;\alpha+1,\beta)
                &=\frac{\alpha}{\alpha+\beta}B(x;\alpha,\beta)
                    -\frac{x^{\alpha}(1-x)^{\beta}}{\alpha+\beta},\\
                B(x;1,\beta)&=\frac{1}{\beta}\left(1-(1-x)^{\beta}\right),\\
                B\left(\frac{1}{4},\frac{1}{2},\frac{1}{2}\right)&=\frac{\pi}{3},
            \end{align}
        for any real numbers $\alpha$, $\beta>0$.
    \end{proof}

Finally, some specific examples are given.
\begin{ex}
    \begin{align}
        &\zeta_{HCB}(1, 1)=\frac{\pi}{3\sqrt{3}},
        &&\zeta_{HCB}(-3,2)=\frac{17}{6}+\frac{74}{81\sqrt{3}}\pi,\\
        &\zeta_{HCB}\left(1,\frac{3}{2}\right)=\left(-\frac{1}{2}+\frac{\sqrt{3}}{3}\right)\pi,
        &&\zeta_{HCB}\left(-2,\frac{7}{2}\right)=\left(-\frac{935}{2048}+\frac{10}{27}\sqrt{3}\right)\pi.
    \end{align}
\end{ex}

\section*{Acknowledgement} 
The authors would like to express their sincere gratitude to Professors Masanobu Kaneko and Yasuo Ohno for their helpful advice and comments. The authors also thank Professor Toshiki Matsusaka for his comments regarding the early stages of the study (in particular, when $a=1/2$). The first author was supported by WISE program (MEXT) at Kyushu University.

\end{document}